\documentclass[12pt, reqno]{amsart}
\usepackage{amscd,amsmath,amsthm,amssymb,graphics}
\usepackage{amsfonts,amssymb,amscd,amsmath,enumitem,verbatim}
\usepackage[a4paper,top=3cm,left=3cm,right=3cm]{geometry}
\theoremstyle{plain}
\usepackage{orcidlink}
\usepackage{color}
\usepackage{hyperref}
\newtheorem{Theorem}{Theorem}
\newtheorem{Lemma}[Theorem]{Lemma}
\newtheorem{Corollary}[Theorem]{Corollary}
\newtheorem{Proposition}[Theorem]{Proposition}

\theoremstyle{definition}

\newtheorem{Remark}[Theorem]{Remark}

\usepackage{graphicx}

\newtheorem{innercustomgeneric}{\customgenericname}
\providecommand{\customgenericname}{}
\newcommand{\newcustomtheorem}[2]{%
  \newenvironment{#1}[1]
  {%
   \renewcommand\customgenericname{#2}%
   \renewcommand\theinnercustomgeneric{##1}%
   \innercustomgeneric
  }
  {\endinnercustomgeneric}
}

\newcustomtheorem{customthm}{Theorem}

\title[Transcendence of $p$-adic continued fractions]{Transcendence of $p$-adic continued fractions and a quantitative $p$-adic Roth theorem}

\author[A. Kalitzin]{Anne Kalitzin}
\address{Dipartimento di Matematica, Università di Trento, Via Sommarive 14, 38123, Povo (TN), Italy}
\email{anne.kalitzin@studenti.unitn.it }

\author[N. Murru]{Nadir Murru}
\address{Dipartimento di Matematica, Università di Trento, Via Sommarive 14, 38123, Povo (TN), Italy}
\email{nadir.murru@unitn.it}

\begin{document}

\begin{abstract}

In this paper, we improve some transcendence results for $p$--adic continued fractions. In particular, we prove that palindromic and quasi--periodic $p$--adic continued fractions converge either to transcendental numbers or quadratic irrationals, removing any restriction on the $p$--adic norm of the partial quotients (or convergents) considered in other works. Moreover, we provide a quantitative version of Ridout's theorem (the $p$--adic analogue of Roth's theorem), and we study the growth of denominators of convergents of algebraic numbers, establishing a $p$--adic version of a well--known result of Davenport and Roth.

\end{abstract}

\maketitle

\section{Introduction}

Several criteria ensuring that a continued fraction converges to a transcendental number have been established, typically under suitable assumptions on its partial quotients. For example, it is not difficult to show that if the partial quotients are unbounded and grow sufficiently rapidly, then the corresponding continued fraction represents a transcendental number \cite{Liou}. On the other hand, transcendence results have also been obtained in the case of bounded partial quotients, especially in connection with so-called quasi-periodic continued fractions. The first contributions in this direction are due to Maillet \cite{Mai} and Baker \cite{Bak}. These results were subsequently improved and generalized by several authors \cite{AB05, AB1, AB2, AB07b, ADQZ, Bug, Dav, Hone}. In particular, in \cite{AB2} the authors proved that a continued fraction beginning with arbitrarily long palindromes converges to a transcendental number. In \cite{Bug}, the authors provide transcendental continued fractions where the partial quotients form certain automatic sequences. 

More recently, these investigations have been extended to the field of $p$--adic numbers $\mathbb Q_p$. Several notions of continued fractions in $\mathbb Q_p$ have been introduced, the most important ones being due to Ruban \cite{Ruban} and Browkin \cite{Browkin}. Thus, it is natural to exploit $p$--adic continued fractions as a tool for constructing transcendental numbers in $\mathbb Q_p$. One of the first main results in this direction is due to Ooto \cite{Ooto} who studied Ruban continued fractions and proved the transcendence of certain quasi--periodic expansions, providing a $p$--adic counterpart of the classical results of Maillet and Baker. In \cite{LMS}, the authors established analogous statements for Browkin continued fractions, covering both the quasi-periodic and the palindromic cases. Then, \cite{CT} deals with general $p$--adic continued fractions and proves transcendence results when the partial quotients are given by certain automatic sequences. In this paper, we improve some results of \cite{LMS} and \cite{CT}.

In \cite{LMS}, the authors proved that the Browkin continued fraction $[0, b_1, b_2, \ldots]$ is transcendental or quadratic irrational under the assumptions that the sequence $(b_1, b_2, \ldots)$ begins with arbitrarily long palindromes, $|b_i|_\infty>C$ (for some constant $C>0$) and $|b_i|_p \geq p^4$ for all $i \gg 0$, where $|\cdot|_\infty$ is the Archimedean norm and $|\cdot|_p$ is the $p$--adic norm. This result has been improved in \cite{CT}, where the last two conditions are removed and only $|B_i|_p^{1/i}$ bounded is required, where $(B_i)_{i \geq 0}$ is the sequence of denominators of convergents. In the Ruban's case, the further condition $|b_i|_p \geq p^2$ is required.
In Section \ref{sec:main}, we are able to remove any restriction on the $p$--adic norm, proving the following theorem which holds for both Browkin and Ruban continued fractions.

\begin{customthm}{A}
Consider $\alpha \in \mathbb{Q}_{p}$ with $p$--adic continued fraction expansion 
\(\alpha=\left[0, b_{1}, b_{2}, \ldots\right],\)
where $\left(b_{i}\right)_{i \geq 1}$ is a sequence beginning with arbitrarily long palindromes. If $\max(|A_i|_\infty^{1/i}, |B_i|_\infty^{1/i}) < p^{1/4}$ for all $i\gg0$, then $\alpha$ is either transcendental or quadratic irrational.  
\end{customthm}

We recall the definition of quasi--periodic continued fractions. Let us consider the sequences of positive integers $(\lambda_i)_{i\geq 0}$, $(k_i)_{i\geq 0}$ and $(n_i)_{i\geq 0}$, with $(n_i)_{i\geq 0}$ strictly increasing and $\lambda_i\geq 2$. Let us suppose that the sequence $(b_n)_{n \geq 0}$ satisfies
\begin{equation*}
b_{m+k_i}=b_m, \ \ \text{        for } n_i\leq m\leq n_i+(\lambda_i-1)k_i-1.
\end{equation*}
The continued fraction $[b_0, b_1, \ldots]$ is called \textit{quasi--periodic}.
In \cite{LMS}, the authors proved that a quasi--periodic Browkin continued fractions $[0, b_1, b_2, \ldots]$ is transcendental or quadratic irrational under a very specific assumption of boundedness for $|b_i|_\infty$ and that $k_i$'s are bounded, $n_{i+1} \geq n_i + \lambda_i k_i$, $\lambda_i > C n_i$ (for a positive constant $C$) for all $i \gg 0$. A consequence of the results proved in \cite{CT} is that if $\alpha=[0, b_1, b_2, \ldots]$ contains arbitrarily long repetitions and $|b_i|_p \geq p^3$ (for Browkin's continued fractions) or $|b_i|_p \geq p^4$ (for Ruban's continued fractions), then $\alpha$ is transcendental or quadratic irrational. In section \ref{sec:main}, we prove a result more close to the one proved by Baker (\cite[Theorem 1]{Bak}).

\begin{customthm}{B}
Let $\alpha=[0,b_1,b_2,\dots]$ be a quasi--periodic $p$-adic continued fraction. 
Assume that for $i \gg 0$ the convergent $A_i/B_i$
satisfies \(
|A_{i}|_\infty \le |B_{i}|_\infty.
\)
If $k_i < C n_i$, for $i \gg 0$ and some constant $C>0$, and 
\begin{equation*} 
    \lim_{i \rightarrow \infty} \cfrac{\log \lambda_i (\log n_i)^{1/2}}{n_i} = \infty,
\end{equation*}
then $\alpha$ is transcendental or quadratic irrational.
\end{customthm}

In the classical real setting, the proofs of the results mentioned at the beginning rely primarily on Roth’s Theorem and the Subspace Theorem, together with their quantitative versions. In the $p$--adic setting, Ridout proved a $p$-adic version of Roth’s Theorem and Schlickewei obtained the corresponding analogue of the Subspace Theorem. To derive the results on quasi--periodic continued fractions, a quantitative version of Roth's theorem is necessary. However, to the best of our knowledge, no quantitative version of Ridout’s Theorem is currently available in a form suitable for our purposes. Thus, one of our main results presented in Section \ref{sec:main} is such a quantitative version which also yields a useful sufficient condition for the transcendence of $p$--adic continued fractions. Specifically, given $\alpha \in \mathbb Q_p$ algebraic of degree $\geq 2$ we prove that
\begin{itemize}

\item for every $0 < \epsilon \leq 1/3$, the number of solutions of
\[ \left| \alpha - \cfrac{A}{B} \right|_p < \cfrac{1}{|B|^{2+\epsilon}_\infty} \]
with $A, B \in \mathbb Z$, $B > 0$, $\gcd(A,B) = 1$ and $|B|_\infty \geq |A|_\infty$ is less than exp$(C_1 \epsilon^{-2})$, with $C_1$ positive constant depending only on $\alpha$;

\item let $(A_i)_{i \geq 0}$ and $(B_i)_{i \geq 0}$ be the numerators and denominators of convergents of $\alpha$, where
\begin{align*}
A_{i}=\frac{\tilde{A}_{i}}{p^{e_{i}}}, \quad B_{i}=\frac{\tilde{B}_{i}}{p^{f_{i}}},
\end{align*}
with $\tilde{A}_{i}, \tilde{B}_{i} \in \mathbb{Z}$, $\gcd(\tilde{A}_i, \tilde{B}_i)=1$ and $p \nmid \tilde{A}_i\tilde{B}_i$. Assuming that
$f_i \ge e_i$ and $|\tilde{A}_i p^{f_i - e_i}|_\infty \le |\tilde{B}_i|_\infty$, then
\[ \log \log |B_i|_p <  \cfrac{C_2 i}{(\log i)^{1/2}}\]
with $C_2$ positive constant depending only on $\alpha$.

\end{itemize}

\section{Preliminaries}

\subsection{Diophantine approximation in \texorpdfstring{$\mathbb Q_p$}{Qp}}

In Diophantine approximation, classical and important results show that algebraic numbers cannot have infinitely many 'good' approximations. The first result in this sense is due to Liouville, who proved that 
\[ \left| \alpha - \frac{p}{q} \right|_\infty > \frac{C}{q^d} \]
for all coprime integers $p$ and $q$, where $\alpha \in \mathbb R$ is an algebraic number of degree $d$ and where $C$ is a positive constant depending only on $\alpha$. As an immediate consequence, it is possible to prove that, for every $\epsilon > 0$, the inequality 
\[ \left| \alpha - \frac{p}{q} \right|_\infty < \frac{1}{q^{d+\epsilon}} \]
is satisfied by finitely many coprime integers $p$ and $q$. This result was improved by Thue, Siegel, Dyson, culminating with the celebrated Roth theorem \cite{Roth}.

\begin{Theorem}
Let $\alpha \in \mathbb R$ be an algebraic number of degree $\geq 2$. Then, for every $\epsilon > 0$, the inequality
\[ \left| \alpha - \frac{p}{q} \right|_\infty < \frac{1}{q^{2+\epsilon}} \]
has only finitely many solutions in the integers $p$ and $q$, with $q > 0$.
\end{Theorem}

Subsequently, Roth’s theorem was extended to higher dimensions in the equally celebrated Schmidt Subspace Theorem \cite{subspace}.

\begin{Theorem}
Let $L_1, \ldots, L_n$ be linear forms, linearly independent over $\mathbb Q$, in $n$ variables with real algebraic coefficients. Then, for every $\epsilon > 0$, the solutions $(x_1, \ldots, x_n) \in \mathbb Z^n$ of
\[ 0 < \prod_{i=1}^n |L_i(x_1, \ldots, x_n)|_\infty < \frac{1}{\max(|x_1|_\infty, \ldots, |x_n|_\infty)^{\epsilon}} \]
belong to a finite union of proper linear subspaces of $\mathbb Q^n$.
\end{Theorem}

All these results have counterparts in the field of $p$--adic numbers $\mathbb Q_p$. We recall that $\mathbb Q_p$ is the completion of $\mathbb Q$ with respect to the $p$--adic norm $| \cdot |_p$, where $p$ is a given prime number. 

The $p$--adic norm of a rational number $x = a/b$ is defined as
\[ |x|_p := p^{-v_p(x)}, \]
where $v_p(\cdot)$ is the $p$--adic valuation, i.e.,
$v_p(x) = v_p(a) - v_p(b)$
and the $p$--adic valuation of an integer is the largest exponent of $p$ dividing it. Thus, a $p$--adic number can be expressed as a formal power series
\[\sum_{i=r}^{+\infty} a_i p^i,\]
for some $r \in \mathbb Z$ and $a_i$ in a complete set of representatives modulo $p$. In 1958, Ridout \cite{Ridout} proved a $p$--adic version of the Roth theorem.

\begin{Theorem} \label{thm:Ridout}
Given $f(X) \in \mathbb Z[X]$ of degree $\geq 2$, let $\alpha$ be a real root of $f(X)$ and $\alpha_i$ be $p_i$--adic roots of $f(X)$ for $i = 1, \ldots, t$ (with $p_1, \ldots, p_t$ distinct primes). Then, for every $\epsilon > 0$, the inequality
\[ \min \left(1, \left| \alpha - \frac{A}{B} \right|_\infty\right) \prod_{i=1}^t \min \left(1, \left| \alpha_i - \frac{A}{B} \right|_{p_i}\right) < \frac{1}{\max(|A|_\infty,|B|_\infty)^{2+\epsilon}} \]
has only finitely many integer solutions $A$ and $B$ with $\gcd(A,B)=1$ and $B > 0$.
\end{Theorem}

Finally, we also have the $p$--adic subspace theorem due to Schlickewei \cite{Sch}. We give here the following formulation (see, e.g., \cite{Bombieri}).

\begin{Theorem} \label{thm:psubspace}
Let $S= \{p, \infty\}$ be a finite set of places. Given $n \geq 2$, consider $L_{1, \infty}, \ldots, L_{n, \infty}$ independent linear forms in $n$ variables with real algebraic coefficients and $L_{1, p}, \ldots, L_{n, p}$ independent linear forms in $n$ variables with $p$-adic algebraic coefficients. Then, for every $\epsilon>0$ there are a finite number of proper rational subspaces such that every $(x_1, \ldots, x_n) \in \mathbb{Z}\left[p^{-1}\right]^{n} \setminus \{0\}$ satisfying
$$
\prod_{v \in S}\prod_{i = 1}^n\left|L_{i, v}(x_1, \ldots, x_n)\right|_v <  \frac{1}{\max\left\{\left|x_{1}\right|_{\infty}, \ldots,\left|x_{n}\right|_{\infty}\right\}^{\epsilon} \max \left\{\left|x_{1}\right|_{p}, \ldots,\left|x_{n}\right|_{p}\right\}^{\epsilon}}
$$
lies in one of these subspaces.
\end{Theorem}

\subsection{\texorpdfstring{$p$}{p}--adic continued fractions}

Classical continued fractions can also be defined in $\mathbb Q_p$. Indeed, it is sufficient to introduce a suitalbe $p$--adic floor function and then apply the same algorithm used over the real numbers. However, in $\mathbb Q_p$, there is no canonical definition of floor function, and for this reason several different definitions of $p$--adic continued fractions have been proposed. The most classical ones are the continued fractions introduced by Ruban \cite{Ruban} and Browkin \cite{Browkin}, both of which are based on the following function:
\[ s: \mathbb Q_p \rightarrow \mathbb Q, \quad s(\alpha) = \sum_{i=r}^0 a_i p^i.\]
The difference between the two approaches lie in the choice of the digits $a_i$. In Ruban's definition they belong to $\{0, \ldots, p-1\}$, while in Browkin's case they are taken in $\{-\frac{p-1}{2}, \ldots, \frac{p-1}{2}\}$ (in the following we always assume $p$ odd). This small difference leads to significantly different behaviours. For instance, with Browkin's definition, every rational number has a finite expansion, whereas in Ruban's case rational numbers may have either finite or periodic expansions. Moreover, it is well known that the analogue of Lagrange's theorem fails in Ruban's continued fraxtions, while for Browkin's continued fractions this is still an open problem. 
In this paper, our results hold for both the Browkin and the Ruban definitions. Thus, given $\alpha_0 = \sum_{i=r}^{+\infty} a_i p^i$ in $\mathbb Q_p$ with $a_i \in \{-\frac{p-1}{2}, \ldots \frac{p-1}{2}\}$ (Browkin) or $a_i \in \{0, \ldots, p-1\}$ (Ruban), the continued fraction expansion of $\alpha_0$ is $[b_0, b_1, \ldots]$, where the partial quotients are evaluated by
\[ \begin{cases} b_i = s(\alpha_i) \cr \alpha_{i+1} = \cfrac{1}{\alpha_i - b_i}   \end{cases}\]
for $i = 0, 1, \ldots$. Let us observe that $v_p(\alpha_i) = v_p(b_i) < 0$  or equivalently $|\alpha_i|_p = |b_i|_p \geq p$ for all $i \geq 1$. We define the convergents of a continued fraction as 
$$\frac{A_i}{B_i} := [b_0, \ldots, b_i], \quad i = 0, 1, \ldots$$
where the sequences $(A_i)_{i \geq 0}$ and $(B_i)_{i \geq 0}$ satisfy the following recurrences:

\begin{align*}
\left\{\begin{array}{l}
A_{-2}=0, \quad A_{-1}=1\\
A_{i}=b_{i} A_{i-1}+A_{i-2}, \ \forall i \geq 0
\end{array}, \quad\left\{\begin{array}{l}
B_{-2}=1, \quad B_{-1}=0 \\
B_{i}=b_{i} B_{i-1}+B_{i-2}, \ \forall i \geq 0.
\end{array}\right.\right.
\end{align*}

Moreover, the classical identities for continued fractions hold:
\[ \alpha_0 = \cfrac{\alpha_{i+1} A_i + A_{i-1}}{\alpha_{i+1} B_i + B_{i-1}}, \quad A_{i} B_{i-1} - B_{i} A_{i-1} = (-1)^i. \]

It is easy to prove by induction that
\begin{equation} \label{eq:vB} 
v_p(B_i) = \sum_{j=1}^i v_p(b_j), \quad |B_i|_p = \prod_{j=1}^i |b_j|_p, \quad \forall i \geq 1. 
\end{equation}
Similarly, if $b_0 = 0$, we have
\begin{equation} \label{eq:vA0}
v_p(A_i) = \sum_{j=2}^i v_p(b_j), \quad |A_i|_p = \prod_{j=2}^i |b_j|_p, \quad \forall i \geq 2.
\end{equation}
If $b_0 \not=0$, then
\[v_p(A_i) = \sum_{j=0}^i v_p(b_j), \quad |A_i|_p = \prod_{j=0}^i |b_j|_p, \quad \forall i \geq 0.\]

In the following proposition, we summarize some further properties that will be useful in the paper.

\begin{Proposition} \label{prop:several}
Given $\alpha \in \mathbb Q_p$, consider the sequence of convergents $(A_i/B_i)_{i \geq -2}$ of its $p$--adic continued fraction expansion, then
\begin{enumerate}
\item $\left| \alpha - \cfrac{A_i}{B_i} \right|_p < \cfrac{1}{\left| B_i\right|_p^2}$, for all $i \geq 0$, see \cite{Browkin, Ruban}.
\item If $\beta \in \mathbb Q_p$ has a $p$--adic continued fraction expansion where the first $i+1$ partial quotients are equal to those of $\alpha$, then $|\alpha-\beta|_p < \cfrac{1}{|B_i|_p^2}$, see \cite{LMS, Ooto}.
\item For all $i \geq-2$, we have $\left|B_{i}\right|_{\infty} \leq\left|B_{i}\right|_{p}$. Moreover, if $\left|b_{0}\right|_{p} \neq 1$, then $\left|A_{i}\right|_{\infty} \leq\left|A_{i}\right|_{p}$ for all $i \geq-2$, see \cite{LMS, Ooto}.
\end{enumerate}
\end{Proposition}

For further information and references about $p$--adic continued fractions see \cite{Romeo}.

\section{Main results} \label{sec:main}

\subsection{Palindromic \texorpdfstring{$p$}{p}--adic continued fractions}
In this section, we prove that a $p$--adic continued fraction starting with arbitrarily long palindromes is transcendental or quadratic irrational. Unlike previous results on this topic in \cite{CT, LMS}, here we are able to remove any constraint on the $p$--adic norm of partial quotients, requiring only that the Archimedean norms of the numerators 
and denominators of the $n$th convergents grow strictly slower 
than $p^{n/4}$ for sufficiently large $n$.
\begin{Lemma}\label{lemma:a2}
Consider $\alpha \in \mathbb{Q}_{p}$ with $p$--adic continued fraction expansion 
\(\alpha=\left[0, b_{1}, b_{2}, \ldots\right],\)
where $\left(b_{n}\right)_{n \geq 1}$ is a sequence beginning with arbitrarily long palindromes. For infinitely many $n$, we have
\begin{align*}
\left|\alpha^{2}-\frac{A_{n-1}}{B_{n}}\right|_{p}
<\frac{\left|b_{1}\right|_{p}}{\left|B_{n}\right|_{p}^{2}}.
\end{align*}
\end{Lemma}

\begin{proof}
By classical results on continued fractions, we have 
$$
M_{n}:=\left(\begin{array}{cc}
B_{n} & B_{n-1} \\
A_{n} & A_{n-1}
\end{array}\right)=\left(\begin{array}{cc}
b_{1} & 1 \\
1 & 0
\end{array}\right)\left(\begin{array}{cc}
b_{2} & 1 \\
1 & 0
\end{array}\right) \ldots\left(\begin{array}{cc}
b_{n} & 1 \\
1 & 0
\end{array}\right) .
$$
The matrix $M_{n}$ is symmetrical if and only if $\left(b_{1}, \ldots, b_{n}\right)$ is palindromic, thus for infinitely many $n$ we have $A_n = B_{n-1}$. 
Recalling that $B_n A_{n-1} - A_n B_{n-1} = (-1)^n$, we obtain
\begin{align*}
\left|\alpha^{2}-\frac{A_{n-1}}{B_{n}}\right|_{p} & =\left|\alpha^{2}-\frac{A_{n-1}}{B_{n-1}} \frac{A_{n}}{B_{n}}\right|_{p}=\left|\left(\alpha+\frac{A_{n-1}}{B_{n-1}}\right)\left(\alpha-\frac{A_{n}}{B_{n}}\right)-\frac{(-1)^{n} \alpha}{B_{n} B_{n-1}}\right|_{p} \\
& \leq \max \left\{\left|\alpha+\frac{A_{n-1}}{B_{n-1}}\right|_{p}\left|\alpha-\frac{A_{n}}{B_{n}}\right|_{p} , \frac{|\alpha|_{p}}{\left|B_{n} B_{n-1}\right|_{p}}\right\}.
\end{align*}
Indeed, since $|\alpha_1|_p = |b_1|_p$, we have 
\begin{align*}
\left|\alpha \right|_p = \left| \frac{1}{\alpha_1}\right|_p = \frac{1}{\left|\alpha_1\right|_p} = \frac{1}{\left|b_1\right|_p} = \left| \frac{A_{n-1}}{B_{n-1}}\right|_p,  
\end{align*}
and
\begin{align*}
\left|\alpha+\frac{A_{n-1}}{B_{n-1}}\right|_{p}\left|\alpha-\frac{A_{n}}{B_{n}}\right|_{p} &< \left|\alpha+\frac{A_{n-1}}{B_{n-1}}\right|_{p} \cdot \frac{1}{|B_n|_p^2}\\
&\le \max \left\{ |\alpha|_p, \left|\frac{A_{n-1}}{B_{n-1}}\right|_{p} \right\} \cdot \frac{1}{|B_n|_p^2}\\
&= |\alpha|_p \cdot \frac{1}{|B_n|_p^2}.
\end{align*}
Consequently,
\begin{align*}
\left|\alpha^{2}-\frac{A_{n-1}}{B_{n}}\right|_{p}
& \leq \max \left\{\left|\frac{\alpha}{B_{n}^{2}}\right|_{p} , \frac{|\alpha|_{p}}{\left|B_{n} B_{n-1}\right|_{p}}\right\}<\frac{\left|b_{1}\right|_{p}}{\left|B_{n}\right|_{p}^{2}}.
\end{align*}

\end{proof}

\begin{Theorem}
Consider $\alpha \in \mathbb{Q}_{p}$ with $p$--adic continued fraction expansion 
\(\alpha=\left[0, b_{1}, b_{2}, \ldots\right],\)
where $\left(b_{n}\right)_{n \geq 1}$ is a sequence beginning with arbitrarily long palindromes. If $\max(|A_n|_\infty^{1/n}, |B_n|_\infty^{1/n}) < p^{1/4}$ for all $n\gg0$, then $\alpha$ is either transcendental or quadratic irrational.    
\end{Theorem}

\begin{proof}
Suppose by contradiction that $\alpha$ is algebraic of degree $\ge 3$.
The following linear forms
$$
L_{1, p}(X, Y, Z)=\alpha Y - X \quad L_{2, p}(X, Y, Z) = Z - \alpha X, \quad L_{3, p}(X, Y, Z) = Y
$$
and
$$
L_{1, \infty}(X, Y, Z) = X, \quad L_{2, \infty}(X, Y, Z)= Y, \quad L_{3, \infty}(X, Y, Z) = Z.
$$
are linearly independent with algebraic coefficients.
Let $0 < \epsilon <1/5$ and $S=\{\infty, p\}$. We want to prove that, for infinitely many $n$, $(A_n, B_n, A_{n-1}) \in \mathbb{Z}\left[p^{-1}\right]^{3} \setminus \{0\}$ satisfies
\begin{align} \label{eq:Lprod}
\prod_{v \in S} \prod_{i = 1}^{3}\left|L_{i, v}(A_n, B_n, A_{n-1})\right|_{v}<\max \left\{\left|A_n\right|_{\infty}, \left|B_n\right|_{\infty}, \left|A_{n-1}\right|_{\infty}\right\}^{-\epsilon} \max \left\{\left|A_n\right|_{p}, \left|B_n\right|_{p},\left|A_{n-1}\right|_{p}\right\}^{-\epsilon}.
\end{align}
For infinitely many $n$, the sequence $(b_1,\ldots,b_n)$ is palindromic and consequently  $A_n = B_{n-1}$. Moreover, from \eqref{eq:vB} and \eqref{eq:vA0}, we also have $|A_n|_p < |B_n|_p$.

It is well--known that
$$
\left|\alpha-\frac{A_{n}}{B_{n}}\right|_{p}<\frac{1}{\left|B_{n}\right|_{p}^{2}},
$$
from which
\begin{align*}
\left| B_n \alpha - A_n\right|_p < |B_n|_p^{-1} \le p^{-n}.
\end{align*}
Moreover, using $A_n = B_{n-1}$ and $B_n A_{n-1} - A_n B_{n-1} = (-1)^n$, we get
\begin{align*}
A_{n-1} = \frac{(-1)^n + B_{n-1}^2}{B_n},   
\end{align*}
and
\begin{align*}
\left| A_{n-1} - \alpha A_n\right|_p &= \left|
\frac{(-1)^n + B_{n-1}^2}{B_n} - \alpha A_n\right|_p =
\left|
B_{n-1}\left( \frac{B_{n-1}}{B_n} - \alpha\right) + \frac{(-1)^n}{B_n}\right|_p \\ 
&\le \max\left( \left|
B_{n-1}\left( \frac{B_{n-1}}{B_n} - \alpha\right)\right|_p, \frac{1}{\left|B_n\right|_p} \right)\le \left|B_n\right|_p^{-1} \le p^{-n}. 
\end{align*}
Thus, we have
\begin{align*}
& |A_n B_n A_{n-1}|
  |B_n \alpha - A_n|_p
  |A_{n-1} - \alpha A_n|_p
  |B_n|_p
  \max\left\{ |A_n|, |B_n|, |A_{n-1}| \right\}^{\epsilon}
  \max\left\{ |A_n|_p, |B_n|_p, |A_{n-1}|_p \right\}^{\epsilon} \\
&< |A_n B_n A_{n-1}|
   |B_n|_p^{-1} |B_n|_p^{-1}
   |B_n|_p
   \max\left\{ |A_n|, |B_n|, |A_{n-1}| \right\}^{\epsilon}
   |B_n|_p^{\epsilon}
   < \frac{p^{(3+\epsilon)n/4}}{p^{(1-\epsilon)n}}
   < 1
\end{align*}
for sufficiently large $n$ and considering that $0 < \epsilon < 1/5$.
By Theorem \ref{thm:psubspace}, all the points $(A_n, B_n, A_{n-1})$ satisfying \eqref{eq:Lprod} lie in a proper rational subspace of $\mathbb{Q}^3$, that is, there exist $Q_1, Q_2, Q_3 \in \mathbb{Q}$ such that
\begin{align*}
Q_1 A_n + Q_2 B_n + Q_3 A_{n-1} = 0.
\end{align*}
Hence,
\begin{align*}
Q_1 \frac{A_n}{B_n} + Q_2 + Q_3 \frac{A_{n-1}}{B_n} = 0.
\end{align*}
By Lemma \ref{lemma:a2}, we obtain 
\begin{align*}
Q_1 \alpha + Q_2 + Q_3 \alpha^2 = 0,
\end{align*}
contradicting $\alpha$ algebraic of degree $\ge 3$.
Thus, $\alpha$ is either transcendental or a quadratic irrational.
\end{proof}

\subsection{A quantitative version of \texorpdfstring{$p$}{p}--adic Roth's theorem}

In Theorem \ref{thm:Ridout}, Ridout assumes that the polynomial $f(X)$ has a real root, which is used to define the Archimedean absolute value in the inequality. However, if we consider a $p$-adic algebraic number $\alpha$, its minimal polynomial over $\mathbb{Q}$ always has at least one complex root. 
Taking the usual complex absolute value as such a root in place of the real root allows the same argument to go through. Thus, in the followong we only assume $\alpha$ a $p$--adic algebraic number.

The proof by Ridout is very similar to the proof of the classical Roth's theorem and consequently we can derive a quantitative version of Theorem \ref{thm:Ridout}, similarly to the quantitative version of Roth's theorem obtained in \cite{DR}. 

The core idea in Ridout's proof is to consider $m \in \mathbb{Z}$ sufficiently large and $\delta \in \mathbb R$ sufficiently small such that
\begin{equation} \label{eq:ridout}
\begin{cases}
 0<\delta<m^{-1}, 
   & \text{(see \cite{Ridout}, Eq. (8))}\\[7pt]
 \cfrac{2m(1 + 5 \delta)}{m - 4 (1+3\delta)nm^{1/2} - 2\eta} < 2 + \epsilon, 
   & \text{(see \cite{Ridout}, pag. 6)}\\[12pt]
 2 \eta + 4(1+3 \delta) n m^{1 / 2}< m,
   & \text{(see \cite{Ridout}, Eq. (9))}
\end{cases}
\end{equation}

where, \(\eta = 10^m \delta^{(1 / 2)^{m}}\). Then, Ridout proves that for every $\epsilon > 0$, the inequality
\begin{align}\label{eq:Rid}
\min (1, \left|\alpha - \frac{A}{B}\right|_\infty) \cdot   
\min (1, |A - B \alpha|_p) < \frac{1}{\max(|A|_\infty, |B|_\infty)^{2+\epsilon}},
\end{align}
can not have $m$ solutions $\frac{A_1}{B_1}, \dots, \frac{A_m}{B_m}$ with $A_i, B_i \in \mathbb{Z}$, $B_i >0$ and $\gcd(A_i,B_i)=1$, that satisfy simultaneously 
\begin{equation} \label{eq:ridout2}
\begin{cases}
\log |B_1|_\infty > \hat{C}m \delta^{-2}, &  \text{(see \cite{Ridout}, Eq. (11))}\\[7pt]
\cfrac{\log |B_j|_\infty}{\log \max(|A_{j-1}|_\infty, |B_{j-1}|_\infty)} > \cfrac{2}{\delta}, \qquad j= 2, \dots, m & \text{(see \cite{Ridout}, Eq. (26))}.
\end{cases}
\end{equation}
where
\(
\hat{C}=2 + 2 \log (2+\bar{A}),
\)
and $\bar{A}=\max \left(1,\left|a_{1}\right|, \ldots,\left|a_{n}\right|\right)$, with $f(x) = x^n + a_1x^{n-1} + \dots + a_n$ the minimal polynomial of $\alpha$.

We can observe that if \eqref{eq:Rid} can not have $m$ solutions, also
\begin{align}\label{eq:Rid2}
\left|\alpha - \frac{A}{B}\right|_p < \frac{1}{\max(|A|_\infty, |B|_\infty)^{2+\epsilon}}
\end{align}
can not have $m$ solutions, because all the solutions of \eqref{eq:Rid2} are solutions of \eqref{eq:Rid}, since
\begin{align*}
|A - B \alpha|_p \le \left|\alpha - \frac{A}{B}\right|_p.
\end{align*}

In the following we focus on the case
$|B|_\infty = \max(|A|_\infty,|B|_\infty)$. 

\begin{Theorem}\label{THm p adic Quantitative Ridout}
Let $\alpha \in \mathbb{Q}_p$ be algebraic of degree $n \ge 2$. 
Let 
\begin{align*}
f(x) = x^n + a_1x^{n-1} + \dots + a_n 
\end{align*} 
be the minimal polynomial of $\alpha$, where $a_{1}, \ldots, a_{n}$ are rational integers. Let
\begin{equation*}
\hat{C}=2 + 2 \log (2+\bar{A}),
\end{equation*}
where $\bar{A}=\max \left(1,\left|a_{1}\right|, \ldots,\left|a_{n}\right|\right)$. 
Then, for every $0 < \epsilon \leq 1/3$, the number of solutions of
\begin{align}\label{p adic quantitative1}
\left|\alpha - \frac{A}{B}\right|_p < \frac{1}{ 2|B|_\infty^{2+\epsilon}},
\end{align}
with $A,B \in \mathbb{Z}$, $B>0$, $\gcd(A,B)=1$ and $|B|_\infty \ge |A|_\infty$, 
is less than 
\begin{align*}
2 \epsilon^{-1} \log \hat{C} +\exp(Cn^2\epsilon^{-2})  
\end{align*} 
for some constant $C$.
\end{Theorem}

\begin{proof}
Consider $m=\left\lfloor 100 n^{2} \epsilon^{-2}\right\rfloor+1$ and $\delta$ such that $\eta = 1$,
where $\eta= 10^m \delta^{(1 / 2)^{m}}$ as above.
First of all, we prove that these choices for $m$ and $\delta$ satisfy equations \eqref{eq:ridout}. 

Since $\eta = 1$, we have $\delta = 10^{- m \cdot 2^m} > 0$ and $0 < \delta < m^{-1}$ is immediately satisfied.

Now, we verify the second equation in \eqref{eq:ridout}. Since 
$\delta = 10^{- m \cdot 2^m}$ and 
$m=\left\lfloor 100 n^{2} \epsilon^{-2}\right\rfloor+1$ with $n \ge 2$ and
$0 < \epsilon \le 1/3$, the smallest value of $m$ is $3601$, and thus $\delta$ is negligible. Hence, we show that
\begin{align*}
\frac{2m}{m - 4 nm^\frac{1}{2} - 2} < 2 + \epsilon,  
\end{align*}
which is equivalent to
\begin{align*}
(2 + \epsilon) (4n \sqrt{m}) + 2 \epsilon + 4 < \epsilon m.  
\end{align*}
Note that
\begin{align*}
(2 + \epsilon) (4n \sqrt{m}) + 2 \epsilon + 4 < 
(2 + \epsilon) (42n^2 \epsilon^{-1}) + 2\epsilon + 4,
\end{align*}
since $m \le 100 n^{2} \epsilon^{-2} +1$ and thus
$\sqrt{m} \le \sqrt{100 n^{2} \epsilon^{-2} +1} \le \sqrt{100 n^{2} \epsilon^{-2}} +1 \le 10.5 n \epsilon^{-1}$.
Using $\epsilon \le 1/3$, we further obtain
\begin{align*}
(2 + \epsilon) (42n^2 \epsilon^{-1}) + 2\epsilon + 4 \le
\frac{294}{3} n^2 \epsilon^{-1} + \frac{14}{3}
< \frac{297}{3} n^2 \epsilon^{-1},
\end{align*}
Since $100 n^2 \epsilon^{-2} \le m$, we finally have
\begin{align*}
\frac{297}{3} n^2 \epsilon^{-1} \le
100 n^2 \epsilon^{-2} \epsilon = \epsilon m.
\end{align*}

Now, we focus on the third equation in \eqref{eq:ridout}. We have
\begin{align*}
2 + 4 (1+3 \delta) n m^{1 / 2}< 2 + 4 (1+3 m^{-1}) n m^{1 / 2} =2 + 4 n m^{1/2} + 12 n m^{-1/2},     
\end{align*}
which is smaller than $m$ if $m \ge 20n^2$. Considering that
$m=\left\lfloor 100 n^{2} \epsilon^{-2}\right\rfloor+1$ with
$0 < \epsilon \le 1/3$, then we are in the case where $m \ge 20n^2$.

Thus, $m$ and $\delta$ satisfy \eqref{eq:ridout}.

Next, consider $\frac{A_{1}}{B_{1}}, \cdots, \frac{A_{m}}{B_{m}}$ with $A_{i}, B_{i} \in \mathbb{Z}$, $\gcd\left(A_{i}, B_{i}\right)=1$ and $0<\left|B_{1}\right|_{\infty}<\left|B_{2}\right|_{\infty}<\ldots$ solutions of \eqref{p adic quantitative1}, that is
$$
\left|\alpha-\frac{A_i}{B_i}\right|_{p}<\frac{1}{2|B_i|_{\infty}^{2+\epsilon}},
$$
where we are assuming $|B_i|_\infty \ge |A_i|_\infty$.
Then
\begin{align*}
\frac{1}{2\left|B_{i} B_{i+1}\right|_{\infty}} &\le 
\frac{1}{\left|B_{i} A_{i+1}\right|_{\infty}+\left|A_{i} B_{i+1}\right|_{\infty}}
\le \frac{1}{\left|B_{i} A_{i+1}-A_{i} B_{i+1}\right|_{\infty}}\\
&\le
\left|B_{i} A_{i+1}-A_{i} B_{i+1}\right|_{p}
\le \left|\frac{B_{i} A_{i+1}-A_{i} B_{i+1}}{B_{i} B_{i+1}}\right|_{p},
\end{align*}
where the last inequality follows from
\[
|B_i B_{i+1}|_p = |B_i|_p\,|B_{i+1}|_p \le 1,
\]
since $B_i,B_{i+1}\in\mathbb{Z}$.
In addition, we have
\begin{align*}
\left|\frac{B_{i} A_{i+1}-A_{i} B_{i+1}}{B_{i} B_{i+1}}\right|_{p} = 
\left|\frac{A_{i+1}}{B_{i+1}}- \frac{A_{i}}{B_{i}}\right|_{p}
\le \max \left(\left|\alpha-\frac{A_{i}}{B_{i}}\right|_{p},
\left|\alpha-\frac{A_{i+1}}{B_{i+1}}\right|_{p}\right)
<\frac{1}{2|B_i|_{\infty}^{2+\epsilon}},
\end{align*}
since $|B_{i+1}|_\infty > |B_{i}|_\infty$.
It follows that
$$
\left|B_{i+1}\right|_{\infty}>\left|B_{i}\right|_{\infty}^{1+\epsilon}
$$
and
\begin{align} \label{p adic repeat}
\frac{\log \left|B_{i+1}\right|_{\infty}}{\log \left|B_{i}\right|_{\infty}}>1+\epsilon.   
\end{align}
Let $k$ be the smallest integer such that
\begin{align*}
(1 + \epsilon)^{k-2} \log 2 > \hat{C}m \delta^{-2},    
\end{align*}
Then, using \eqref{p adic repeat} repeatedly, we obtain  
\begin{align*}
\log |B_k|_\infty > (1 + \epsilon)^{k-2} \log |B_2|_\infty \ge (1 + \epsilon)^{k-2} \log 2 > \hat{C}m \delta^{-2},    
\end{align*}
since $1 \le |B_1|_\infty < |B_2|_\infty$ implies
$|B_2|_\infty \ge 2$.

Let $l$ be the smallest integer such that
\begin{align*}
(1+ \epsilon)^l > \frac{2}{\delta},   
\end{align*}
then
\begin{align*}
\frac{\log \left|B_{i+l}\right|_{\infty}}{\log \left|B_{i}\right|_{\infty}}
> \frac{2}{\delta} \qquad \text{for} \quad i = 1,2,\dots
\end{align*}
Indeed, by \eqref{p adic repeat},
\begin{align*}
\frac{\log \left|B_{i+l}\right|_{\infty}}{\log \left|B_{i}\right|_{\infty}} =
\frac{\log \left|B_{i+l}\right|_{\infty}}{\log \left|B_{i+l-1}\right|_{\infty}} \cdots \frac{\log \left|B_{i+2}\right|_{\infty}}{\log \left|B_{i+1}\right|_{\infty}} \cdot \frac{\log \left|B_{i+1}\right|_{\infty}}{\log \left|B_{i}\right|_{\infty}}
> (1+\epsilon)^l > \frac{2}{\delta} \qquad \text{for} \quad i = 1,2,\dots
\end{align*}

Thus the number of solutions of \eqref{p adic quantitative1} must be less than $k+(m-1)l$, otherwise
$$
b_1 = B_{k}, \quad b_2= B_{k+l}, \quad \dots, \quad b_m =B_{k+(m-1)l}
$$
are $m$ denominators of solutions of \eqref{p adic quantitative1} that satisfy simultaneously both the equations in \eqref{eq:ridout2} which is not possible as proved by Ridout \cite{Ridout}. 
Indeed, we have
\[
\log |b_1|_\infty = \log |B_k|_\infty > \hat{C} m \delta^{-2},
\]
and, for $j=2,\dots,m$,
\[
\frac{\log |b_j|_\infty}{\log |b_{j-1}|_\infty}
= \frac{\log |B_{k+(j-1)l}|_\infty}{\log |B_{k+(j-2)l}|_\infty}
= \frac{\log |B_{i+l}|_\infty}{\log |B_i|_\infty} > \frac{2}{\delta},
\]
where $i = k + (j-2)l$.

Finally, we estimate $k+(m-1)l$. 
Since $l$ is the smallest integer such that
\begin{align*}
(1+ \epsilon)^l > \frac{2}{\delta},   
\end{align*}
we have
\begin{align*}
\log l + \log(1+ \epsilon) > \log 2 + \log \delta^{-1},   
\end{align*}
and thus $l$ is the smallest integer such that
\begin{align*}
l > \frac{ \log 2 + \log \delta^{-1}}{\log(1+ \epsilon)}.   
\end{align*}
Hence,
\begin{align*}
l \le \frac{ \log 2 + \log \delta^{-1}}{\log(1+ \epsilon)} +1 \le \frac{ \log 2 + m 2^m \log10}{\epsilon/2} +1,   
\end{align*}
and consequently there exists a constant $c_1 > 0$ such that
\begin{align*}
l \le
c_1 \cdot \frac{m 2^m}{\epsilon},   
\end{align*}
where $c_1 = 2+ 2 \log(10)$.
Since $k$ is the smallest integer such that
\begin{align*}
(1 + \epsilon)^{k-2} \log 2 > \hat{C}m \delta^{-2},
\end{align*}
we get
\begin{align*}
(k-2) \log (1+\varepsilon)>\log \hat{C}+\log m+2 \log \delta^{-1}-\log \log 2,  
\end{align*}
and solving for $k$ gives
\begin{align*}
k>\frac{\log \hat{C}+\log m+2 \log \delta^{-1}-\log \log 2}{\log (1+ \epsilon)}+2.    
\end{align*}

Since $k$ is the smallest such integer, we have
\begin{align*}
k &\le \frac{\log \hat{C}+\log m+2 \log \delta^{-1}-\log \log 2}{\log (1+ \epsilon)}+3 \\
&\le
\frac{2\log \hat{C}+ m2^m +4 m \cdot 2^m \log10 + m2^m}{\epsilon} =
2\log \hat{C}\epsilon^{-1} + c_2\frac{m2^m}{\epsilon}
\end{align*}
where $c_2 = 2 + 4 \log10$.
Putting all together we obtain
\begin{align*}
k + (m-1)l &\le 2\log \hat{C}\epsilon^{-1} + c_2\frac{m2^m}{\epsilon} + (m-1) c_1 \frac{m2^m}{\epsilon}.
\end{align*}
Thus, it suffices to show that
\begin{align*}
c_2\frac{m2^m}{\epsilon} + (m-1) c_1 \frac{m2^m}{\epsilon} \le \exp(Cn^2\epsilon^{-2})    
\end{align*}
for some constant $C >0$.
We begin by noting that
\begin{align*}
c_2\frac{m2^m}{\epsilon} + (m-1) c_1 \frac{m2^m}{\epsilon} = \frac{m2^m}{\epsilon} (c_2 + (m-1) c_1) \le m^2 2^m\epsilon^{-1}c_3,  
\end{align*}
since $c_2 + (m-1) c_1 =m c_1 + c_2-c_1 \le c_3m$ for $c_3$ sufficiently large. Rewriting in exponential form gives
\begin{align*}
m2^m\epsilon^{-1}c_3m = \exp(2\log m + m \log2 + \log\epsilon^{-1} + \log c_3). 
\end{align*}
Using $m \le 101n^2\epsilon^{-2}$ and $n^2\epsilon^{-2}>1$, we obtain
\begin{align*}
\log m \le \log 101 + \log(n^2\epsilon^{-2}) \le \log 101 + n^2\epsilon^{-2} \le c_4 n^2\epsilon^{-2},   
\end{align*}
where $c_4>0$ is a sufficiently large constant. Since $\epsilon^{-1}\ge 3 > 1$, we have
\begin{align*}
\log\epsilon^{-1} \le  n^2\epsilon^{-2}. 
\end{align*}
Choose $c_5>0$ such that
\begin{align*}
\log c_3 \le c_5 n^2\epsilon^{-2}.    
\end{align*}
We obtain
\begin{align*}
\exp(2\log m + m \log2 + \log\epsilon^{-1} + \log c_3) &\le
\exp(2c_4n^2\epsilon^{-2} + \log2 \cdot 101 n^2\epsilon^{-2} + n^2\epsilon^{-2} + c_5n^2\epsilon^{-2})\\
&= \exp(n^2\epsilon^{-2}(2c_4 + \log2 \cdot 101 + 1 + c_5)). 
\end{align*}
Hence, after setting $C = 2c_4 + \log2 \cdot 101 + 1 + c_5$, we obtain the desired upper bound
\begin{align*}
k + (m-1)l \le 2\log \hat{C}\epsilon^{-1} + \exp(Cn^2\epsilon^{-2}),
\end{align*}
which completes the proof.
\end{proof}

As a consequence, we obtain the following corollary.

\begin{Corollary}\label{coroquanti}
Let $\alpha \in \mathbb{Q}_p$ be algebraic over $\mathbb{Q}$ of degree $n \ge 2$. 
Let 
\begin{align*}
f(x) = x^n + a_1x^{n-1} + \dots + a_n 
\end{align*} be the minimal polynomial of $\alpha$, where $a_{1}, \ldots, a_{n}$ are rational integers. Let
\begin{equation*}
\hat{C}=2 + 2 \log (2+\bar{A}),
\end{equation*}
where $\bar{A}=\max \left(1,\left|a_{1}\right|, \ldots,\left|a_{n}\right|\right)$.
Then, for every $0<\epsilon\le 2/3$, the number of solutions of 
\begin{align}\label{p adic quantitative2}
\left|\alpha - \frac{A}{B}\right|_p < \frac{1}{ |B|_\infty^{2+\epsilon}},
\end{align}
with $A,B \in \mathbb{Z}$, $B>0$, $\gcd(A,B)=1$ and $|B|_\infty \ge |A|_\infty$, 
is less than 
\begin{align*}
4 \epsilon^{-1} \log \hat{C} +\exp(C_1n^2\epsilon^{-2})  
\end{align*} 
for some constant $C_1>0$.
\end{Corollary}

\begin{proof}
Consider $A/B$ solution of \eqref{p adic quantitative2} such that $|B|_{\infty} \ge 4^{\epsilon^{-1}}$, hence $|B|_{\infty}^{\epsilon/2} \ge 2$, then 
\begin{align*}
\left|\alpha - \frac{A}{B}\right|_p < \frac{1}{ |B|_\infty^{2+\epsilon}}
< \frac{1}{ 2|B|_\infty^{2+\epsilon/2}}.
\end{align*}
Thus, $A/B$ is also a solution of \eqref{p adic quantitative1} with $0<\frac{\epsilon}{2} \le \frac{1}{3}$.
Thus, by Theorem \ref{THm p adic Quantitative Ridout}, the number of solutions of \eqref{p adic quantitative2} such that $|B|_{\infty} \ge 4^{\epsilon^{-1}}$ is less than 
\begin{align*}
4 \epsilon^{-1} \log \hat{C} +\exp(c_1n^2\epsilon^{-2}),  
\end{align*} 
where $c_1 = 4C$ with $C$ constant of Theorem \ref{THm p adic Quantitative Ridout}.

On the other hand, the number of solutions of \eqref{p adic quantitative2} with $|B|_{\infty}<4^{\epsilon^{-1}}$ can be bounded by $\exp \left(c_{2} \epsilon^{-1}\right)$ for some constant $c_2$. Indeed, both $A$ and $B$ are integers and are bounded by $4^{\epsilon^{-1}}$.
Hence there are at most $2\cdot 4^{\epsilon^{-1}}+1$ possible values for $A$ and at most
$4^{\epsilon^{-1}}$ possible values for $B$ ($B>0$). Therefore the total number of pairs $(A,B)$ is bounded by
\[
(2\cdot 4^{\epsilon^{-1}}+1) 4^{\epsilon^{-1}}
\le 4 \cdot 4^{2\epsilon^{-1}},
\]
and consequently the number of distinct values of $A/B$ is bounded by $4 \cdot 4^{2\epsilon^{-1}}
\le \exp(c_2 \epsilon^{-1})$ for 
$c_2 = 4\log4$.

Thus, the number of solutions of \eqref{p adic quantitative2} is less than 
\begin{align*}
4 \epsilon^{-1} \log \hat{C} +\exp(c_1n^2\epsilon^{-2}) + \exp(c_2 \epsilon^{-1}),  
\end{align*} 
where
\begin{align*}
\exp(c_1n^2\epsilon^{-2}) + \exp(c_2 \epsilon^{-1}) &\le \exp(c_1n^2\epsilon^{-2}) + \exp(c_2 n^2\epsilon^{-2}) \le 2\exp(c_3n^2\epsilon^{-2})\\
&\le \exp(\log2 + c_3n^2\epsilon^{-2})
\le \exp(C_1n^2\epsilon^{-2}),
\end{align*}
with $c_3 = \max(c_1,c_2)$ and $C_1 = \log2 + c_3$. Consequently,
the number of solutions of \eqref{p adic quantitative2} is less than 
\begin{align*}
4 \epsilon^{-1} \log \hat{C} + \exp(C_1n^2\epsilon^{-2}).  
\end{align*} 
\end{proof}

\begin{Remark}\label{RemarkbetterBound}
Note that we can bound the number of solutions of \eqref{p adic quantitative2} by
\begin{align*}
\exp(c_1\epsilon^{-2}),    
\end{align*}
where \(c_1 = c_1(\alpha)\) depends on \(\alpha\) through \(n\) and \(\bar A\).
Indeed, since $n \ge 2$ and $0 < \epsilon \le 2/3$,
\begin{align*}
4 \epsilon^{-1} \log \hat{C} \le
4 n^2\epsilon^{-2} \log \hat{C}
\le \exp(4 n^2\epsilon^{-2} \log \hat{C}) = \exp(c_2n^2\epsilon^{-2}),  
\end{align*}
where $c_2 = 4\log \hat{C}$.
Therefore,
\begin{align*}
4 \epsilon^{-1} \log \hat{C} + \exp(Cn^2\epsilon^{-2}) \le \exp(c_2n^2\epsilon^{-2}) + \exp(Cn^2\epsilon^{-2}) \le 2\exp(c_3n^2\epsilon^{-2}),  
\end{align*}
with $c_3 = \max(c_2,C)$. 
Finally,
\[
2\exp(c_3 n^2\epsilon^{-2})
= \exp\!\bigl(\log 2 + c_3 n^2\epsilon^{-2}\bigr)
\le \exp(c_1\epsilon^{-2}),
\]
where \(c_1 = n^2\log 2 + c_3\) depends only on \(\alpha\).
\end{Remark}

\subsection{On the growth of denominators of convergents for algebraic numbers in \texorpdfstring{$\mathbb Q_p$}{Qp}}

In this section, we establish a bound on the $p$--adic norm of denominators of convergents in the $p$--adic continued fraction expansion of algebraic numbers, providing a $p$--adic analogue of an important result of Davenport and Roth \cite{DR}. As a consequence, we have a useful sufficient condition for the transcendence of $p$-adic continued fractions.

\begin{Lemma} \label{p lioville}
Let $\alpha \in \mathbb{Q}_{p}$ be algebraic of degree $n \ge 2$ over $\mathbb{Q}$, then there exists $c:=c(\alpha)>0$ such that  $a, b \in \mathbb{Z}, b \neq 0$, we have
$$
\left|\alpha-\frac{a}{b}\right|_{p} \geq \frac{c}{\max \left(|a|_{\infty},|b|_{\infty}\right)^n}.
$$    
\end{Lemma}

\begin{proof}
Let $f(x) \in \mathbb{Z}[x]$ be the minimal polynomial of $\alpha$ over $\mathbb{Q}$.
Since $\mathbb{Q}_p$ has characteristic zero, the Taylor expansion of \(f\)
around \(\alpha\) yields
\begin{align*}
f(x)= f(\alpha) +(x-\alpha) f^{\prime}(\alpha)+(x-\alpha)^{2} \frac{f^{\prime \prime}(\alpha)}{2!}+\ldots+(x-\alpha)^{n} \frac{f^{(n)}(\alpha)}{n!}.    
\end{align*}
Since $f(\alpha) = 0$, we have
\begin{align*}
f\left(\frac{a}{b}\right)= \left(\frac{a}{b}-\alpha\right) \cdot \left( f^{\prime}(\alpha) +\left(\frac{a}{b}-\alpha \right) \frac{f^{\prime \prime}(\alpha)}{2!} + \ldots + \left(\frac{a}{b}-\alpha\right)^{n-1} \frac{f^{(n)}(\alpha)}{n!} \right). 
\end{align*}
We set
\begin{align*}
g(\alpha) = \left(\frac{a}{b}-\alpha \right) \frac{f^{\prime \prime}(\alpha)}{2!} + \ldots + \left(\frac{a}{b}-\alpha\right)^{n-1} \frac{f^{(n)}(\alpha)}{n!}. 
\end{align*}
Then,
\begin{align*}
f\left(\frac{a}{b}\right)= \left(\frac{a}{b}-\alpha\right) \cdot \left( f^{\prime}(\alpha) + g(\alpha) \right), 
\end{align*}
and
\begin{align*}
\left|\frac{a}{b}-\alpha\right|_p =
\frac{\left| f\left(\frac{a}{b}\right) \right|_p}{\left|f^{\prime}(\alpha) + g(\alpha)\right|_p}
\ge \frac{\left| f\left(\frac{a}{b}\right) \right|_p}{\max(\left|f^{\prime}(\alpha)\right|_p, \left|g(\alpha)\right|_p)}. 
\end{align*}
By the ultrametric inequality, we have
\[
|g(\alpha)|_p = \left| \sum_{k=2}^{n} \frac{f^{(k)}(\alpha)}{k!} \left(\frac{a}{b} - \alpha\right)^{k-1} \right|_p 
\le \max_{2 \le k \le n}\left( \left| \frac{f^{(k)}(\alpha)}{k!} \right|_p \, \left|\frac{a}{b} - \alpha\right|_p^{\,k-1}\right).
\]
Now define
\[
C(\alpha) := \max_{2 \le k \le n} \left| \frac{f^{(k)}(\alpha)}{k!} \right|_p, \quad \tilde{c} = \frac{|f^{\prime}(\alpha)|_p}{C(\alpha)}.
\]
If 
$\left| \frac{a}{b} - \alpha \right|_p \ge \tilde{c}$,
we can take $c = \tilde{c}(\alpha)$ since
$$
\left|\frac{a}{b}-\alpha\right|_p \geq \tilde{c} \geq \frac{\tilde{c}}{\max \left(|a|_{\infty},|b|_{\infty}\right)^n} \qquad \text{since} \quad a, b \in \mathbb{Z}.
$$
If 
$\left| \frac{a}{b} - \alpha \right|_p < \tilde{c}$,
\[
|g(\alpha)|_p \le C(\alpha) \, \left|\frac{a}{b} - \alpha\right|_p < |f^{\prime}(\alpha)|_p.
\]
and hence, $\max\left(\left|f^{\prime}(\alpha)\right|_{p},|g(\alpha)|_{p}\right)=\left|f^{\prime}(\alpha)\right|_{p}$, from which
\begin{align*}
\left|\alpha-\frac{a}{b}\right|_{p}=\frac{\left|f\left(\frac{a}{b}\right)\right|_{p}}{\left|f^{\prime}(\alpha)\right|_{p}}.   
\end{align*}
Writing $f(x)=c_{n} x^{n}+c_{n-1} x^{n-1}+\ldots+c_{1} x+c_{0}$, with $c_i \in \mathbb Z$, then
$$
f\left(\frac{a}{b}\right)=\frac{c_n a^{n}+c_{n-1} b a^{n-1}+\ldots+c_{1} b^{n-1} a+c_{0} b^{n}}{b^{n}}.
$$
Denote by \(N \in \mathbb Z\) the numerator of this expression. 
We obtain
\begin{align*}
\left|\alpha-\frac{a}{b}\right|_{p}= & \frac{|f(\frac{a}{b})|_{p}}{\left|f^{\prime}(\alpha)\right|_{p}}
=\frac{|N|_{p}}{|b^n|_{p} \left|f^{\prime}(\alpha)\right|_{p}} 
\geq \frac{1}{|N|_{\infty}\left|f^{\prime}(\alpha)\right|_{p}}.
\end{align*}
Note that
\begin{align*}
|N|_{\infty} &\leq \left|c_d a^{n}\right|_{\infty}+\left|c_{n-1} ba^{n-1}\right|_{\infty}+\dots+\left|c_{0} b^{n}\right|_{\infty}\\
&\leq |c_n|_\infty \max(|a|_\infty, |b|_\infty)^n + \left|c_{n-1}\right|_{\infty} \max(|a|_\infty, |b|_\infty)^n + \dots +  |c_0|_\infty \max(|a|_\infty, |b|_\infty)^n\\
&\leq \hat{c} \max(|a|_\infty, |b|_\infty)^n,
\end{align*}
where $\hat{c} = |c_n|_\infty + \dots + |c_0|_\infty$ depends only on $\alpha$.
Setting $c=\frac{1}{\hat{c} \cdot\left|f^{\prime}(\alpha)\right|_{p}}$, we obtain the thesis:
\begin{align*}
\left|\alpha-\frac{a}{b}\right|_{p}
\geq \frac{c}{\max(|a|_\infty, |b|_\infty)^n}.
\end{align*}
\end{proof}

\begin{Theorem} \label{thm:loglog}
Let $\alpha \in \mathbb{Q}_p$ be algebraic of degree $n\ge2$ and
$A_{k} / B_{k}$ be the convergents of its $p$-adic continued fraction expansion.
Consider 
\begin{align*}
A_{k}=\frac{\tilde{A}_{k}}{p^{e_{k}}}, \quad B_{k}=\frac{\tilde{B}_{k}}{p^{f_{k}}},
\end{align*}
where $\tilde{A}_{k}, \tilde{B}_{k} \in \mathbb{Z}$, $\gcd(\tilde{A}_k, \tilde{B}_k)=1$ and $p \nmid \tilde{A}_k\tilde{B}_k$. Assume that
$f_k \ge e_k$ and $|\tilde{A}_k p^{f_k - e_k}|_\infty \le |\tilde{B}_k|_\infty$.
Then
\begin{equation}\label{(11_new)}
\log \log |B_{k}|_p<\frac{c(\alpha) k}{(\log k)^{1 / 2}},
\end{equation}
where $c(\alpha)$ is independent of $k$.
\end{Theorem}

\begin{proof}
We know that
\begin{equation}\label{(33_new)}
\left|\beta-\frac{A_{k}}{B_{k}}\right|_p = \frac{1}{|B_{k}|_p |B_{k+1}|_p}.
\end{equation}
and for all sufficiently large \(k\), we also have
\(|B_k|_\infty < |B_k|_p.\)
If 
\begin{align*}
|B_{k+1}|_p >|B_{k}|_p^{1+\epsilon},
\end{align*}
then the rational number \(\tilde A_k p^{f_k-e_k}/\tilde B_k\) satisfies
\begin{align}\label{w1}
\left|\alpha - \frac{\tilde{A}_{k}p^{f_k - e_k}}{\tilde{B}_{k}}\right|_p &= \left|\alpha - \frac{A_{k}}{B_{k}}\right|_p = \frac{1}{|B_{k}|_p |B_{k+1}|_p} < \frac{1}{{|B_{k}|_p}^{2 + \epsilon}}\\
&< \frac{1}{{|B_{k}|_\infty}^{2 + \epsilon}} =
\frac{{|p^{f_k}_{k}|_\infty}^{2 + \epsilon}}{{|\tilde{B}_{k}|_\infty}^{2 + \epsilon}}.
\end{align}
In particular, any solution of
\begin{align*}
\left|\alpha - \frac{\tilde{A}_{k}p^{f_k - e_k}}{\tilde{B}_k}\right|_p < \frac{1}{ |\tilde{B}_k|_\infty^{2+\epsilon}},
\end{align*}
also satisfies~\eqref{w1}.
Therefore, by Corollary~\ref{coroquanti} and Remark~\ref{RemarkbetterBound},
the number of indices \(k\) such that
\[
|B_{k+1}|_p > |B_k|_p^{1+\epsilon}
\]
is at most
\[
\exp\!\left(c_1\epsilon^{-2}\right),
\]
where \(c_1>0\) depends only on \(\alpha\).
In particular, there are only finitely many of these 'exceptional' indices. 

Moreover, by Lemma \ref{p lioville}, 
\begin{align*}
\left|\alpha-\frac{A_{k}}{B_{k}}\right|_{p}=\left|\alpha-\frac{\tilde{A}_{k} \cdot\left|B_{k}\right|_{p}}{\tilde{B}_{k} \cdot\left|A_{k}\right|_{p}}\right|_{p} \geq \frac{c}{\max \left(\left|\tilde{A}_{k}\right|_{\infty} \cdot\left|B_{k}\right|_{p},\left|\tilde{B}_{k}\right|_{\infty} \cdot\left|A_{k}\right|_{p}\right)^{n}},
\end{align*}
where $c = c(\alpha)$.
Since $f_k \ge e_k$, we have 
\begin{align*}
\left| \tilde{A}_k\right|_\infty \le \left|\tilde{A}_k p^{f_k - e_k}\right|_\infty \le \left|\tilde{B}_k\right|_\infty.
\end{align*}
By hypothesis, \(|B_k|_p\ge |A_k|_p\), thus we obtain
\begin{align*}
\max \left(\left|\tilde{A}_{k}\right|_{\infty} \cdot\left|B_{k}\right|_{p},\left|\tilde{B}_{k}\right|_{\infty} \cdot\left|A_{k}\right|_{p}\right) &\leq \left|B_{k}\right|_{p} \cdot \max \left(\left|\tilde{A}_{k}\right|_{\infty},\left|\tilde{B}_{k}\right|_{\infty}\right) \\
&\leq \left|B_{k}\right|_{p} \cdot\left|\tilde{B}_{k}\right|_{\infty} = \left|B_{k}\right|_{p} \cdot \left|B_{k}\right|_{\infty} \cdot \left|B_k\right|_{p} \\
&\leq \left|B_{k}\right|_{p}^{3}.
\end{align*}
Combining this with~\eqref{(33_new)}, we deduce
\begin{align*}
\frac{1}{|B_{k}|_p |B_{k+1}|_p} \ge \frac{c}{\left|B_{k}\right|_{p}^{3n}},
\end{align*}
and hence
\begin{align*}
|B_{k+1}|_p \le \frac{1}{c}\left|B_{k}\right|_{p}^{3n-1} \le \left|B_{k}\right|_{p}^{c \prime (3n-1)},
\end{align*}
for a constant \(c'=c'(\alpha)\) sufficiently large.
Therefore,
\begin{align*}
\frac{\log|B_{k+1}|_p}{\log \left|B_{k}\right|_{p}} \le c_2,
\end{align*}
where $c_2 = \log(c \prime (3n-1))$.

Apart from the above exceptional indices, we have
$$
\frac{\log |B_{k+1}|_p}{\log |B_{k}|_p} \leq 1+\epsilon.
$$
Summarizing, we know that

\begin{itemize}
\item $\displaystyle \frac{\log |B_{k+1}|_p}{\log |B_{k}|_p} \le 1 + \epsilon, 
       \qquad \forall\, k \ \text{non-exceptional}$,
\item $\displaystyle \frac{\log |B_{k+1}|_p}{\log |B_{k}|_p} \le c_2, 
       \qquad \forall\, k \ \text{exceptional}$,
\item The number of exceptional $k$ is 
       $\le \exp\!\left( c_1 \epsilon^{-2} \right)$. 
\end{itemize}

Thus, setting $r_i = \log |B_{i+1}|_p/\log |B_i|_p$, we have
\(r_i \le c_2\) for the exceptional values of $i$, $r_i \le 1 + \epsilon$ otherwise.
For all $k > 1$,
\[
\log |B_k|_p = \left(\prod_{i=1}^{k-1} r_i \right) \log |B_1|_p.
\]
Splitting the product into exceptional and non-exceptional factors gives
\[
\prod_{i=1}^{k-1} r_i 
\le (1+\epsilon)^{(k-1)-E(k-1)} \, c_2^{E(k-1)}
\le (1+\epsilon)^{k-1} \, c_2^{E_{\max}},
\]
where $E(k-1)$ is the number of exceptional indices $i \le k-1$, and $E_{\max}$ is the total number of exceptional indices.
Using the previous bound $E_{\max} \le \exp(c_1 \epsilon^{-2})$, we obtain
\[
\log |B_k|_p \le \log |B_1|_p (1+\epsilon)^{k} \, c_2^{\exp(c_1 \epsilon^{-2})}.
\]
Hence,
\[
\log \log |B_{k}|_p < k \log(1+\epsilon) + c_3 \exp\big(c_1 \xi^{-2}\big),
\]
where \(c_3 = \log(c_2)\).  
Choosing $\epsilon = c_4 (\log k)^{-1/2}$ with $c_4$ sufficiently large, such that $c_1 c_4^{-2} < 1$, we obtain
\begin{align*}
\log \log |B_{k}|_p &< k \log(1+ c_4 (\log k)^{-1/2}) + c_3 \exp\big(c_1 c_4^{-2} \log k \big)\\
&= k \log(1+ c_4 (\log k)^{-1/2}) + c_3 k^{c_1 c_4^{-2}}\\
&< k \log(1+ c_4 (\log k)^{-1/2}) + c_3 k.
\end{align*}
Recall that $\lim_{x \to 0} \log(1+x)/x = 1$, hence for sufficiently large $k > c_5$, we can approximate
$\log(1+ c_4 (\log k)^{-1/2})$ with
$c_4 (\log k)^{-1/2}$.
Thus,
\begin{align*}
\log \log |B_{k}|_p &< k c_4 (\log k)^{-1/2} + c_3 k\\
&= k (c_4 (\log k)^{-1/2} + c_3)\\
&< \frac{c(\alpha)k}{(\log k)^{1/2}},
\end{align*}
where $c(\alpha)$ is sufficiently large, that is $c(\alpha) > c_4 + c_3 (\log k)^{-1/2}$.
\end{proof}

\subsection{Quasi--periodic \texorpdfstring{$p$}{p}--adic continued fractions}

 In the following, we prove that a quasi--periodic continued fraction, under certain assumption, converges either to a transcendental number or to a quadratic irrational. The assumptions considered here are more relaxed than those in \cite{CT, LMS}.

\begin{Lemma}\label{bakerlemma3_p}
For all $n \ge 0$, we have
\begin{equation}\label{baker10_p}
|B_n|_p \geq \left(\frac{1}{2}(1+\sqrt{5})\right)^{n-1}
\end{equation}
\end{Lemma}

\begin{proof}
We proceed by induction. Note that $B_0 = 1$ and $|B_0|_p = 1 \ge \left( \frac{1 + \sqrt{5}}{2}\right)^{-1}$. 

Moreover, $B_1 = b_1$, thus 
\begin{align*}
|B_1|_p = |b_1|_p \ge p > 1 = \left( \frac{1 + \sqrt{5}}{2}\right)^{0}.
\end{align*}
Let $k$ be a positive integer and assume \eqref{baker10_p} holds for $n=1,2,\dots,k$.  
We have
$$
\begin{aligned}
|B_{k+1}|_p & = |B_{k}|_p \cdot |b_{k+1}|_p\\
& \geq \left(\frac{1}{2}(1+\sqrt{5})\right)^{k-1}\cdot |b_{k+1}|_p\\
&\ge \left(\frac{1}{2}(1+\sqrt{5})\right)^{k},
\end{aligned}
$$
where in the last equality we used that
$|b_{k+1}|_p \ge p > 2 > \frac{1}{2}(1+\sqrt{5})$.
\end{proof}

\begin{Lemma}\label{lemma:height}
Suppose $\alpha = [0,b_1, b_2, \dots] \in \mathbb{Q}_p$. Let $h$, $k$ be positive integers and
\begin{align}\label{baker9}
\eta = \left[0, b_1, \dots, b_{h-1}, \overline{b_h, \dots, b_{h+k-1}}\right].
\end{align}
Then $\eta$ has naive height less than $2|B_{h+k-1}|_p^2$.
\end{Lemma}

\begin{proof}
Let $\eta_{n}$ be the $n$th complete quotient of $\eta$. 
By construction, the partial quotients satisfy $b_n=b_{n+k}$ for all $n\ge h$, hence $\eta_h=\eta_{h+k}$. By the classical identity for continued fractions
$$
\eta=\frac{\eta_{h} A_{h-1}+A_{h-2}}{\eta_{h} B_{h-1}+B_{h-2}}=\frac{\eta_{h} A_{h+k-1}+A_{h+k-2}}{\eta_{h} B_{h+k-1}+B_{h+k-2}},
$$
where we recall that $A_{-1}=1$ and $B_{-1}=0$.
From these two representations we obtain a quadratic relation
$$
P \eta^{2}+Q \eta+R=0,
$$
where
$$
\begin{aligned}
& P=B_{h-2} B_{h+k-1}-B_{h-1} B_{h+k-2} \\
& Q=B_{h-1} A_{h+k-2}+A_{h-1} B_{h+k-2}-A_{h-2} B_{h+k-1}-B_{h-2} A_{h+k-1} \\
& R=A_{h-2} A_{h+k-1}-A_{h-1} A_{h+k-2}.
\end{aligned}
$$
By Proposition \ref{prop:several},
\begin{align*}
|P|_\infty &\le |B_{h-2}|_\infty |B_{h+k-1}|_\infty + |B_{h-1}|_\infty |B_{h+k-2}|_\infty\\ 
&<
|B_{h-2}|_p |B_{h+k-1}|_p + |B_{h-1}|_p |B_{h+k-2}|_p\\
&\le 2|B_{h+k-1}|_p^2.
\end{align*}
One can easily see 
that
\begin{align*}
0 \le |A_i|_p \le |B_i|_p, \quad \forall i\ge 0. 
\end{align*}
Consequently, each term appearing in $Q,R$ is bounded in absolute value by $|B_{h+k-1}|_p^2$, and hence
$$
\max (|P|_\infty,|Q|_\infty,|R|_\infty) < 2|B_{h+k-1}|_p^2.
$$
\end{proof}

\begin{Theorem}\label{th1baker_p}
Let $\alpha=[0,b_1,b_2,\dots]$ be a quasi--periodic $p$-adic continued fraction. Assume that for $i \gg 0$ the convergent $A_i/B_i$
satisfies \(
|A_{i}|_\infty \le |B_{i}|_\infty.
\)
If $k_i < C n_i$, for $i \gg 0$ and some constant $C>0$, and 
\begin{equation} \label{eq:lim}
    \lim_{i \rightarrow \infty} \cfrac{\log \lambda_i (\log n_i)^{1/2}}{n_i} = \infty,
\end{equation}
then $\alpha$ is transcendental or quadratic irrational.
\end{Theorem}

\begin{proof}
Suppose $\alpha$ is algebraic of degree $r > 2$, we define
\[ \eta_i := [0, b_1, \ldots, b_{n_i-1},\overline{b_{n_i},\ldots, b_{n_i+k_i-1}}]. \]
Then by construction $\alpha$ and $\eta_i$ agree on the first $n_i + \lambda_ik_i$ partial quotients. 
Consider $\eta_i = (p+\sqrt{d}) / q$, then
\begin{align*}
p= \pm Q, \qquad q=\mp 2 P, \quad \text{and} \quad d=Q^{2}-4 P R,   
\end{align*}
where $P \eta_i^2+Q\eta_i+R=0$. By Lemma \ref{lemma:height}, $|P|_\infty,|Q|_\infty,|R|_\infty$ are less than $2 |B_{n_{i}+k_{i}-1}|_p^{2}$. 
Hence, $|p|_\infty$ and $|q|_\infty$ are both less than $4 |B_{n_{i}+k_{i}-1}|_p^{2}$. 
By Theorem \ref{p lioville} we have
\begin{align*}
\left| \alpha - \frac{p + \sqrt{d}}{q} \right|_p =  
\left| \beta - \frac{p}{q} \right|_p \ge \frac{c(\beta)}{\max(|p|_\infty, |q|_\infty)^{d_\beta}},
\end{align*}
where $\beta := \alpha - \frac{\sqrt{d}}{q} \in \mathbb{Q}_p$ and $d_\beta$ denotes the degree of $\beta$ over $\mathbb{Q}$. 

The constant $c(\beta)$ depends on $\alpha$ and $\sqrt{d}$. But since $\sqrt{d}$ comes from $\eta^{(i)}$ which was defined by the partial quotients of $\alpha$, we actually have $c(\beta) = c(\alpha)$. Moreover, note that $d_\beta \le 2r$, hence we obtain
\begin{align*}
\left| \alpha - \frac{p + \sqrt{d}}{q} \right|_p 
\ge \frac{c(\alpha)}{\max(|p|_\infty, |q|_\infty)^{2r}}.
\end{align*}
Together with the upper bounds for $|p|_\infty$ and $|q|_\infty$, we get
\begin{align*}
\left| \alpha - \eta^{(i)} \right|_p 
\ge 
\frac{c_2}{|B_{n_{i}+k_{i}-1}|_p^{4r}},
\end{align*}
where $c_2 = c(\alpha)/ 4^{2r}$.

However, by Proposition \ref{prop:several},
$$
\left|\alpha-\eta^{(i)}\right|<|B_{m_{i}}|_p^{-2},
$$
where $m_{i}=n_{i}+k_{i} \lambda_{i}-1$. 
Hence,
$$
|B_{m_{i}}|_p< c_{3} |B_{n_{i}+k_{i}-1}|_p^{2 r},
$$
where $c_3 = 1/\sqrt{c_2}$.
Using Lemma \ref{bakerlemma3_p} we obtain
$$
\left(\frac{1}{2}(1+\sqrt{5})\right)^{m_{i}-1}<c_{3} |B_{n_{i}+k_{i}-1}|_p^{2 r}.
$$   
Taking the logarithm we obtain
\begin{align*}
(m_i - 1) \log(\frac{1}{2}(1+\sqrt{5}))< \log(c_{3}) +  2r \log(|B_{n_{i}+k_{i}-1}|_p).
\end{align*}
and
$$ 
m_{i}-1< \frac{\log(c_3)}{\log(\frac{1}{2}(1+\sqrt{5}))} + c_{4} \log |B_{n_{i}+k_{i}-1}|_p,
$$
where $c_4 = 2r / \log(\frac{1}{2}(1+\sqrt{5}))$.
Thus, for all sufficiently large $i$,
$$
\lambda_{i}<m_{i}-1<c_{4} \log |B_{n_{i}+k_{i}-1}|_p,
$$
that is,
$$
\log \lambda_{i}<c_{5} \log \log |B_{n_{i}+k_{i}-1}|_p,
$$
where we choose $c_5$ such that $\log(c_4) + \log \log |B_{n_{i}+k_{i}-1}|_p < c_{5} \log \log |B_{n_{i}+k_{i}-1}|_p$.

\medskip
By Theorem \ref{thm:loglog}, we obtain for $i \gg 0$,
$$
\log \lambda_{i}<c_{6}\left(n_{i}+k_{i}-1\right)\left\{\log \left(n_{i}+k_{i}-1\right)\right\}^{-\frac{1}{2}},
$$
where $c_6$ depends only on $\alpha$.
Using the condition $k_{i}<C n_{i}$ it follows that
$$
\log \lambda_{i}<c_{7} n_{i}\left(\log n_{i}\right)^{-\frac{1}{2}}
$$
for $i \gg 0$, where $c_7 = c_6 \cdot (C+1)$.
This contradicts \eqref{eq:lim} and the theorem is proved.
\end{proof}

\end{document}